\documentclass{amsart}
\usepackage{epsfig}

\newtheorem{theorem}{Theorem}[section]
\newtheorem{lemma}[theorem]{Lemma}

\newtheorem{problem}[theorem]{Question}
\newtheorem{corollary}[theorem]{Corollary}
\newtheorem{proposition}[theorem]{Proposition}

\theoremstyle{definition}

\newtheorem{note}[theorem]{Note}

\theoremstyle{remark}

\begin{document}

\title[Zagier polynomials. Arithmetic properties]
{The Zagier polynomials. Part II: Arithmetic
properties of coefficients}

\author[M. Coffey et al.]{Mark W. Coffey}
\address{Department of Physics,
Colorado School of Mines, Golden, CO 80401}
\email{mcoffey@mines.edu}

\author[]{Valerio De Angelis}
\address{Department of Mathematics,
Xavier University of Louisiana, New Orleans, LA 70125}
\email{vdeangel@xula.edu}

\author[]{Atul Dixit}
\address{Department of Mathematics,
Tulane University, New Orleans, LA 70118}
\email{adixit@tulane.edu}

\author[]{Victor H. Moll}
\address{Department of Mathematics,
Tulane University, New Orleans, LA 70118}
\email{vhm@tulane.edu}

\author[]{Armin Straub}
\address{Department of Mathematics,
University of Illinois at Urbana-Champaign, Urbana, IL 61801}
\email{astraub@illinois.edu}

\author[]{Christophe Vignat}
\address{Department of Mathematics,
Tulane University, New Orleans, LA 70118 and
L.S.S. Supelec, Universite d'Orsay, France}
\email{vignat@tulane.edu}

%    General info

\date{\today}


\begin{abstract}
The modified Bernoulli numbers
\begin{equation*}
B_{n}^{*} = \sum_{r=0}^{n} \binom{n+r}{2r} \frac{B_{r}}{n+r}, \quad n > 0
\end{equation*}
introduced by D. Zagier in $1998$ were recently 
extended to the polynomial case by
replacing $B_{r}$ by the Bernoulli polynomials $B_{r}(x)$. Arithmetic
properties of the coefficients of these polynomials are established. In
particular, the $2$-adic valuation of the modified Bernoulli numbers is
determined. A variety of analytic, umbral, and asymptotic methods is used to 
analyze these polynomials.
\end{abstract}

\maketitle

\newcommand{\ba}{\begin{eqnarray}}
\newcommand{\ea}{\end{eqnarray}}
\newcommand{\ift}{\int_{0}^{\infty}}
\newcommand{\nn}{\nonumber}
\newcommand{\no}{\noindent}
\newcommand{\A}{\mathfrak{A}}
\newcommand{\B}{\mathfrak{B}}
\newcommand{\C}{\mathfrak{C}}
\newcommand{\D}{\mathfrak{D}}
\newcommand{\E}{\mathfrak{E}}
\newcommand{\pe}{\mathfrak{P}}
\newcommand{\Q}{\mathfrak{Q}}
\newcommand{\U}{\mathfrak{U}}
\newcommand{\lf}{\left\lfloor}
\newcommand{\rf}{\right\rfloor}
\newcommand{\vm}{{}}
\newcommand{\realpart}{\mathop{\rm Re}\nolimits}
\newcommand{\imagpart}{\mathop{\rm Im}\nolimits}

\newcommand{\ZZ}{\mathbb{Z}}
\newcommand{\op}[1]{\ensuremath{\operatorname{#1}}}
\newcommand{\ontop}[2]{\ensuremath{\genfrac{}{}{0pt}{}{#1}{#2}}}
\newcommand{\pFq}[5]{\ensuremath{{}_{#1}F_{#2} \left( \genfrac{}{}{0pt}{}{#3}
{#4} \bigg| {#5} \right)}}

\newtheorem{Definition}{\bf Definition}[section]
\newtheorem{Thm}[Definition]{\bf Theorem}
\newtheorem{Example}[Definition]{\bf Example}
\newtheorem{Lem}[Definition]{\bf Lemma}
\newtheorem{Cor}[Definition]{\bf Corollary}
\newtheorem{Prop}[Definition]{\bf Proposition}
\numberwithin{equation}{section}

\section{Introduction}
\label{sec-intro}

The Bernoulli numbers $B_{n}$, defined by the generating function
\begin{equation}
\frac{t}{e^{t}-1} = \sum_{n=0}^{\infty} B_{n} \frac{t^{n}}{n!},
\label{bn-def}
\end{equation}
\noindent
were extended by D. Zagier \cite{zagier-1998a} with the introduction of the
so-called \textit{modified Bernoulli numbers} $B_{n}^{*}$ defined by
\begin{equation}
B_{n}^{*} = \sum_{r=0}^{n} \binom{n+r}{2r} \frac{B_{r}}{n+r}.
\label{bn-star}
\end{equation}
\noindent
Note that $B_{0}^{*}$ is undefined. Arithmetic
properties of $B_{2n}$ ($B_{1} = - \tfrac{1}{2}$ and $B_{2n+1} = 0$, for
$n > 0$), include
the von Staudt--Clausen theorem which states that, for $n>0$,
\begin{equation}
B_{2n} \equiv - \sum_{\ontop{(p-1) | 2n}{\text{$p$ prime}}} \frac{1}{p}
\quad \bmod 1.
\end{equation}
\noindent
It follows that the denominator of $B_{2n}$ is the
product of all primes $p$ such that $p-1$ divides $2n$. On the other hand,
the numerators of $B_{2n}$ are still a mysterious sequence.

The definition \eqref{bn-star} shows that $B_{n}^{*}$ is a rational
number. Write it in reduced form and define
\begin{equation}\label{def-alpha}
\alpha_{n} = \op{denom}(B_{n}^{*}).
\end{equation}
\noindent
Zagier \cite{zagier-1998a} showed that
\begin{equation}
\tilde{B}_{n} = 2nB_{n}^{*} - B_{n}
\end{equation}
\noindent
satisfies
\begin{equation}
\tilde{B}_{n} \equiv \sum_{\ontop{(p+1) | n}{p \text{ prime }}}
\frac{1}{p} \quad \bmod 1, \quad (n > 0, \, n \text{ even})
\end{equation}
\noindent
that implies
\begin{equation}
2nB_{n}^{*}  \equiv - \sum_{\ontop{(p-1) | n}{p \text{ prime }}}
\frac{1}{p}  +
\sum_{\ontop{(p+1) | n}{p \text{ prime }}}
\frac{1}{p}
\quad \bmod 1, \quad (n > 0).
\end{equation}
\noindent
This statement shows that if $p$ is a prime dividing $\alpha_{n}$ (defined in
\eqref{def-alpha}), then at least one of $p$, $p-1$ and $p+1$ divides $n$. In
particular, all prime factors  $p$ of $\alpha_{n}$ satisfy $p \leq n+1$. In
fact, from computing the first $1000$ terms, it appears that, conjecturally,
the following stronger statement is true: if $p$ is a prime dividing
$\alpha_{n}$, then $p+1$ or $p-1$ divides $n$.

The first few values of the sequence $\{B_{n}^{*} \}$ are
\begin{align*}
  \frac{3}{4},\frac{1}{24},-\frac{1}{4},-\frac{27}{80},-\frac{1}{4},
  -\frac{29}{1260},\frac{1}{4},\frac{451}{1120},\frac{1}{4},-\frac{65}{264},\ldots
\end{align*}

Our particular interest will be in the $2$-adic properties of this sequence and
the $2$-adic valuation of $B_n^*$ will be worked out completely.
A guiding question motivated by the first few terms as above is:

\begin{problem}\label{q:main}
Is the denominator $\alpha_{n}$ always divisible by $4$?
\end{problem}

This basic question will become particularly relevant when considering the
corresponding modifications of Bernoulli polynomials. This is addressed at the
end of this introduction.

It turns out that $\alpha_{2n+1} = 4$, so only even indices need to be
considered. The first few values of
$\tfrac{1}{4}\alpha_{2n}$ are given by
\begin{equation}
6, \, 20, \, 315, \, 280, \, 66, \, 3003, \, 78, \, 9520, \,
305235, \, 20900, \, 138, \, 19734, \, 6, \, 7540, \ldots
\end{equation}

This sequence has been recently added to OEIS (the
database created by N. Sloane) as entry $\text{A}216912$. The 
next figure shows the $2$-adic valuation
of $\alpha_{2n}$; that is, the highest power of $2$ that divides $\alpha_{2n}$.

\begin{figure}[htbp]
  \begin{center}
    \includegraphics[width=0.55\textwidth]{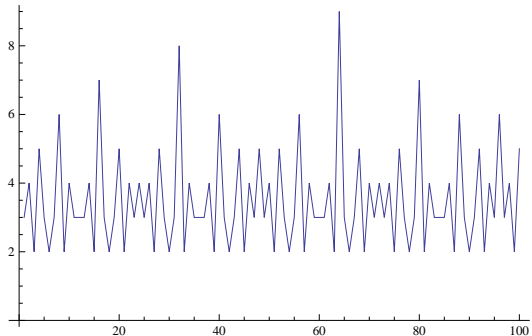}
    \caption{Power of $2$ that divides denominator of $B_{2n}^{*}$}
    \label{den-2a}
  \end{center}
\end{figure}

Symbolic computations lead us to discover the next result.
In particular, this answers Question \ref{q:main} in the affirmative.

\begin{theorem}
\label{thm:2adic}
  For $n>0$,
  \begin{align*}
    \nu_2(\alpha_n) = -\nu_2(B^*_n) = 2 + \nu_2(n)
    - \begin{cases} 1 & \text{if $n \equiv 6 \bmod 12$}, \\
      2 & \text{if $n \equiv 0 \bmod 12$}, \\
      0 & \text{otherwise.} \end{cases}
  \end{align*}
In particular, 
$B_{n}^{*}$, the denominator of $\alpha_{n}$, is divisible by $4$.
\end{theorem}

Note that this may be rephrased in the following way: The $2$-adic valuations $\nu_2(8nB^*_{2n})$
form a periodic sequence of period $6$ with values
\begin{equation}
 \left\{ 0, \, 0, \, 1, \, 0, \, 0, \, 2 \right\}.
\end{equation}
This is an unexpected variation on the \textit{period $6$ theme}: D. Zagier
proved that the sequence $\{ B_{2n+1}^{*} \}$ is $6$-periodic.

The modified Bernoulli numbers $B_{n}^{*}$ were extended in \cite{dixit-2012a}
to the Zagier polynomials defined by
\begin{equation}
B_{n}^{*}(x) = \sum_{r=0}^{n} \binom{n+r}{2r} \frac{B_{r}(x)}{n+r},
\label{bn-star-poly}
\end{equation}
\noindent
so that $B_{n}^{*} = B_{n}^{*}(0)$. The first few are:
\begin{align*}
  &\frac{1}{4} (2 x+3),
  \frac{1}{24} \left(6 x^2+18 x+1\right),
  \frac{1}{12} (2 x+3) \left(x^2+3 x-1\right),\\
  &\quad \frac{1}{80} \left(10 x^4+60 x^3+90 x^2-27\right),
  \frac{1}{60} (2 x+3) \left(3 x^4+18 x^3+23 x^2-12 x-5\right), \ldots
\end{align*}
In analogy to $\alpha_{n}$ in \eqref{def-alpha}, define, for $j \in \ZZ$,
\begin{equation}\label{def-alphan}
\alpha_{n,j} = \op{denom}(B_{n}^{*}(j)).
\end{equation}
It is shown in Lemma \ref{lem:den4} of Section \ref{sec:alphan} that, under the
assumption that $4$ divides $\alpha_n$, the denominators $\alpha_{n,j}$ equal
$\alpha_n$ for any $j \in \ZZ$.  Combining this with Theorem \ref{thm:2adic},
one obtains:

\begin{theorem}
\label{thm:alphan}
The denominator $\alpha_{n,j} = \op{denom}(B_{n}^{*}(j))$ does not
depend on the value $j\in\ZZ$.
\end{theorem}

Special values of $B_{n}^{*}(x)$ present interesting arithmetic properties.
The relation
\begin{equation}
B_{n}^{*}(x+1) = B_{n}^{*}(x) + \frac{1}{2}U_{n-1} \left( \frac{x}{2} + 1
\right),
\end{equation}
\noindent
relating $B_{n}^{*}$ to the Chebyshev polynomial of the second kind, appears as
Lemma $10.2$ in \cite{dixit-2012a}. In particular, this shows the identity
\begin{equation}
B_{n}^{*}(1) = B_{n}^{*} + \frac{n}{2}.
\end{equation}
\noindent
On the other hand, the values $B_{n}^{*}(-1)$ are connected to the
asymptotic expansion of the function
\begin{equation}
V(z) = \log z + \psi \left( z + \frac{1}{z} \right)
\end{equation}
\noindent
at $z \to 0$. Here, $\psi(z)$ is the digamma function
\begin{equation}
\psi(z) = \frac{\Gamma'(z)}{\Gamma(z)},
\end{equation}
the logarithmic derivative of the gamma function. The proof of the
next statement appears
in Section \ref{sec-zagier}.

\begin{theorem}
Define the numbers $v_{n}$ by the asymptotic expansion
\begin{equation}
V(z) \sim \sum_{n=0}^{\infty} v_{n} z^{n}.
\end{equation}
\noindent
Then $v_{n} = -2B_{n}^{*}(-1)$.
\end{theorem}

The value $v_{2n-1} = (-1)^{n}/2$ is simple to obtain, but
\begin{equation} 
v_{2n}  =  (-1)^{n+1} \left[
\frac{1}{n} + \sum_{k=1}^{n} (-1)^{k}
\binom{n+k-1}{n-k} \frac{B_{2k}}{2k} \right]
\end{equation}
\noindent
requires further work.

A second motivation for considering the sequence
$\{ v_{n} \}$ comes from the natural interest in the sequence $\{B_{2n}^{*} \}$.
The established fact that $\{ B_{2n+1}^{*} \}$ is $6$-periodic has no
obvious analog for the even indices. It turns out that the function
$V(z)$ satisfies
\begin{equation}
\label{genfunB-even}
\sum_{n=1}^{\infty} B_{2n}^{*}z^{2n} = - \frac{1}{2} V(z) -
\frac{z}{4} \left[ \frac{1}{z^{2}+1} + \frac{2(1-z^{4})}{1-z^{6}} \right],
\end{equation}
\noindent
thus connecting $B_{2n}^{*}$ and $v_{n}$.

A variety of expressions for the coefficients $v_{n}$ are provided. 
Section \ref{sec-auxiliary} gives one using the 
umbral method and Section \ref{sec-zagier}
exploits a relation between the Zagier polynomials $B_{n}^{*}$ and the 
Chebyshev polynomials $U_{n}(x)$ to determine $v_{n}$. A direct 
asymptotic method is used in 
Section \ref{sec-glasser} and Section \ref{sec-asymptotics} presents a 
family of polynomials that determine $v_{n}$. The classical integral 
representation of the digamma function is used in Section \ref{sec-int-rep}, 
the formula of Fa\`{a} di Bruno to differentiate compositions is used in 
Section \ref{sec-faa} and, finally, a recurrence for $v_{n}$ is 
analyzed in Section \ref{sec-vn} by the WZ-method \cite{petkovsek-1996a}.

\section{The $2$-adic valuation of $B^*_n$}
\label{sec:2adic}

The goal of this section is to establish Theorem \ref{thm:2adic} which
determines the $2$-adic valuation of the sequence $B^*_n$.

The strategy employed here is as follows. It is a consequence of the von
Staudt--Clausen congruence that the Bernoulli numbers $2B_n$ are $2$-integral.
From this one may conclude that the rational numbers $4nB^*_n$ are
$2$-integral as well. In particular, these numbers can be reduced modulo
powers of $2$ to determine their $2$-adic valuation.
Here, it will be sufficient to reduce them 
modulo $8$. To begin with, the classical
Bernoulli numbers are reduced modulo $8$.

\begin{proposition}
\label{prop:Bmod8}
  The following congruences hold modulo $8$:
  \[ 2 B_0 \equiv 2, \hspace{1em} 2 B_2 \equiv 3, \hspace{1em} 2 B_{2 k}
     \equiv \left\{ \begin{array}{ll}
       1 & \text{if $k$ even},\\
       5 & \text{if $k$ odd},
     \end{array} \right. \]
  with $k > 1$.
\end{proposition}

\begin{proof}
The von Staudt--Clausen theorem states that 
\begin{equation}
pB_{2k} \equiv p-1 \bmod p^{\ell +1}
\end{equation}
\noindent
for $p$ prime, $k \geq 2$ and when $(p-1)p^{\ell}$ divides $2k$; see 
\cite{olver-2010a}, formula $24.10.2$ on page 593. Now 
take $p=2$ and $\ell =1$. Then for $k \geq 2$ it follows that 
$2B_{2k} \equiv 1 \bmod 4$. Therefore 
$2B_{2k} \equiv 1 \text{ or } 5 \bmod 8$. In the case $k$ is even, one may 
take $\ell = 2$, since then $(p-1)p^{\ell} = 4$ divides $2k$. Therefore
\begin{equation}
2B_{2k} \equiv 1 \bmod 8.
\end{equation}
\noindent 
A different proof of this fact appears in \cite{carlitz-1957a}. The identity 
established there is 
\begin{equation}
2B_{2k} \equiv 1 \bmod 2^{r+1}
\end{equation}
\noindent
where $2^{r}$ is the highest power of $2$ that divides $2k$. In particular, for 
$k$ even, $r \geq 2$ and the result follows. 

The case $k$ odd requires a different approach. \\

  Let $U_m$ be the numerator and $V_m$ the denominator of $B_m$, so that $B_m
  = U_m / V_m$ and $(U_m, V_m) = 1$, $V_m > 0$. Voronoi's congruence
  \cite[Proposition 15.2.3]{ireland-1990a} states 
that, if $m \geq 2$ is even and
  $a$, $n$ are positive integers with $(a, n) = 1$, then
  \[ (a^m - 1) U_m \equiv m a^{m - 1} V_m \sum_{j = 1}^{n - 1} j^{m - 1}
     \left[ \frac{j a}{n} \right] \hspace{1em} \bmod n.
  \]
  As usual, $[x]$ refers to the greatest  integer less than or equal to $x$.
  It follows from the von Staudt--Clausen congruence that $2 B_{2 m}$ has
  $2$-adic valuation $0$ for $m > 0$, so that they are $2$-integral.
  Voronoi's congruence with $a = 3$ and $n = 64$ therefore yields
  \begin{eqnarray*}
    (3^m - 1) 2 B_m & \equiv & 2 m \, 3^{m - 1} \sum_{j = 1}^{63} j^{m - 1}
    \left[ \frac{3 j}{64} \right] \hspace{1em} \bmod 64.
  \end{eqnarray*}
  One easily checks that, for even $m$, $3^m - 1 \equiv 4 m$ modulo $64$.
  Similarly, after checking finitely many cases, for $m \equiv 2$ modulo $4$
  with $m \geq 6$,
  \[ 3^{m - 1} \sum_{j = 1}^{63} j^{m - 1} \left[ \frac{3 j}{64} \right]
     \equiv 42 \bmod 64. \]
  Combining these, one finds, for $m \equiv 2$ modulo $4$ with $m \geq 6$,
  \begin{eqnarray*}
    2 B_m  \frac{m}{2} & \equiv & 5 \frac{m}{2} \hspace{1em} \bmod 8.
  \end{eqnarray*}
  Hence, if $m = 2 k$ with $k \geq 3$ odd, then $2 B_m \equiv 5$ modulo
  $8$.
\end{proof}

Further basic ingredients are the following generating functions.

\begin{proposition}\label{prop:genfuns}
The following generating functions admit rational closed-forms:
  \begin{align}
  2 + \sum_{n = 1}^{\infty} x^n \sum_{k = 0}^n \binom{n + k}{2 k}  \frac{2 n}{n
    + k} & = \frac{2 - 3 x}{1 - 3 x + x^2},  \label{eq:genfun1}\\
    2 + \sum_{n = 1}^{\infty} x^n \sum_{k=0}^{\lf n/2 \rf} \binom{n + 2 k}{4 k}  \frac{2 n}{n +
    2 k} & = \frac{\left( 1 - 2 x \right) \left( 2 - 2 x + x^2
    \right)}{\left( 1 - x + x^2 \right) \left( 1 - 3 x + x^2 \right)},
    \nonumber\\
    2 + \sum_{n = 1}^{\infty} x^n \sum_{k=0}^{\lf n/2 \rf} 
\left( - 1 \right)^k \binom{n + 2
    k}{4 k}  \frac{2 n}{n + 2 k} & = \frac{2 - 6 x + 7 x^2 - 2 x^3}{1 - 4 x
    + 7 x^2 - 4 x^3 + x^4} . \nonumber
  \end{align}
\end{proposition}
\begin{proof}
These readily follow from the generating function for $T_{n}(x)$, the 
Chebyshev polynomials of the first kind, given by 
\begin{equation}
\sum_{n=0}^{\infty} T_{n}(x) t^{n} = \frac{1 - xt}{1 - 2xt + t^{2}}
\end{equation}
\noindent
and from the fact 
\begin{equation}
\sum_{r=0}^{n} \binom{n+r}{2r} \frac{x^{r}}{n+r} = \frac{1}{n} 
T_{n} \left( \frac{x}{2} +1 \right)
\end{equation}
\noindent
proved as Lemma $9.1$ in \cite{dixit-2012a}.
\end{proof}

Equipped as such, a proof of Theorem \ref{thm:2adic} is given next. The 
statement of this theorem is repeated for the 
convenience of the reader.

\begin{theorem}\label{thm:2adic2}
  For $n>0$,
  \begin{align*}
    % \nu_2(\alpha_n) =
    -\nu_2(B^*_n) = 2 + \nu_2(n)
    - \begin{cases} 1 & \text{if $n \equiv 6 \bmod 12$}, \\
      2 & \text{if $n \equiv 0 \bmod 12$}, \\
      0 & \text{otherwise}. \end{cases}
  \end{align*}
\end{theorem}

\begin{proof}
It is convenient to remark at the beginning that the case of odd $n$ is
simple and is a consequence of Zagier's result on the periodicity of
the sequence $B^*_{2n+1}$.

Working modulo $8$ and using Proposition \ref{prop:Bmod8}, it follows that
$2 B_0 \equiv 2$, $2 B_1 \equiv -1$, $2 B_2 \equiv 3$ and for $k > 1$,
  \begin{align*}
    2 B_{2 k} \equiv 3 - 2 \left( - 1 \right)^k .
  \end{align*}

Note that $\binom{n + k}{2 k}  \frac{2 n}{n + k}$ is an integer. Thus it
follows from \eqref{bn-star} that
$4n B^*_n$ is a $2$-adic integer. For $n \geq 1$, these numbers reduce
modulo $8$ to
  \begin{align*}
    4 n B_n^{\ast} & = \sum_{k=0}^{\lf n/2 \rf} 
\binom{n + k}{2 k}  \frac{2 n}{n + k} 2 B_k\\
    & = - n^2 + \sum_{k=0}^{\lf n/2 \rf}
 \binom{n + 2 k}{4 k}  \frac{2 n}{n + 2 k} 2 B_{2k}\\
    & \equiv - n^2 + 2 - n \binom{n + 1}{3} + 
\sum_{k=0}^{\lf n/2 \rf}  \left[ 3 - 2 \left( -
    1 \right)^k \right] \binom{n + 2 k}{4 k}  \frac{2 n}{n + 2 k},
  \end{align*}
where in the second equality, the $- n^2$ term comes from 
the contribution of $B_1
  = - 1 / 2$, the only nonzero Bernoulli number of odd index. Also, for 
the final congruence, adjusting for the $k = 0$ and $k = 1$ cases in
  which $2 B_0 = 2$ and $2 B_2 = 1 / 3 \equiv 3$ respectively, produces the
  extra terms
  \begin{align*}
    \binom{n}{0}  \frac{2 n}{n}  \left( 2 B_0 - 1 \right) + \binom{n + 2}{4}
    \frac{2 n}{n + 2}  \left( 2 B_2 - 5 \right)
    % = 2 + \frac{n}{2} \binom{n +  1}{3} \left( 2 B_2 - 5 \right)
    \equiv 2 - n \binom{n + 1}{3} .
  \end{align*}
  Using Proposition \ref{prop:genfuns} modulo $8$ now gives
  \begin{align*}
    4 + \sum_{n = 1}^{\infty} 4 n B_n^{\ast} x^n
    & \equiv \frac{2}{1 - x} -
    \frac{x \left( 1 + x \right) \left( 1 + x^2 \right)}{\left( 1 - x
    \right)^5} + 3 \frac{\left( 1 - 2 x \right) \left( 2 - 2 x + x^2
    \right)}{\left( 1 - x + x^2 \right) \left( 1 - 3 x + x^2 \right)} \\
    &\quad - 2 \frac{2 - 6 x + 7 x^2 - 2 x^3}{1 - 4 x + 7 x^2 - 4 x^3 + x^4},
  \end{align*}
  where it is readily verified that the right-hand side is a rational function
  whose coefficients modulo $8$ are periodic with period $24$. The even part
  simplifies to
  \begin{align*}
    \sum_{n = 1}^{\infty} 8 n B_{2 n}^{\ast} x^{2n} \equiv \frac{x \left( 3 + x
    + 6 x^2 + x^3 + 3 x^4 + 4 x^5 \right)}{1 - x^6} .
  \end{align*}
  This implies
  \begin{align*}
    \nu_2 \left( 8 n B_{2 n}^{\ast} \right) = \left\{ \begin{array}{ll}
      0 & \text{if $(n,3)=1,$}\\
      1 & \text{if $n \equiv 3 \hspace{1em} \bmod 6, $}\\
      2 & \text{if $n \equiv 0 \hspace{1em} \bmod 6, $}
    \end{array} \right.
  \end{align*}
  which proves the claim.
\end{proof}

\section{The denominators of $B_{n}^{*}(j)$}
\label{sec-div4}
\label{sec:alphan}

The goal of this section is to establish Theorem \ref{thm:alphan}. It
states that the denominator of $B^*_n(j)$ does not depend on $j\in\ZZ$. The
proof begins with the identity
\begin{equation}
B_{n}^{*}(x+1) = B_{n}^{*}(x) + \frac{1}{2}U_{n-1} \left( \frac{x}{2} + 1
\right),
\label{cheby-zag0}
\end{equation}
\noindent
appearing as Lemma $10.2$ in \cite{dixit-2012a} which establishes a relation
between the Zagier polynomials and the Chebyshev polynomials of the 
second kind $U_{n}(x)$.

\begin{lemma}
  \label{lemma-UnZ}
For every half-integer $x$, the numbers $U_{n}(x)$ are integers.
\end{lemma}
\begin{proof}
  This is clear upon using the determinant representation
  \begin{align}\label{eq:Udet}
    U_n (x) = \left|\begin{array}{cccc}
      2 x & 1 &  & 0\\
      1 & 2 x & \ddots & \\
      & \ddots & \ddots & 1\\
      0 &  & 1 & 2 x
    \end{array}\right|
  \end{align}
\noindent
for the Chebyshev polynomial.  To verify \eqref{eq:Udet} denote the determinant
by $D_{n}(x)$. By expansion by minors, it follows that $D_{n+1}(x) = 2xD_{n}(x)
- D_{n-1}(x)$. The same recurrence is satisfied by $U_{n}(x)$ and a direct
computation gives $D_{n}(x) = U_{n}(x)$ for $n =1, \, 2$. Thus, $U_{n}(x) =
D_{n}(x)$ for all $n \in \mathbb{N}$.

\smallskip

An alternative proof employs the
generating function of the $U_{n}(x)$ polynomials
\begin{equation}
\sum_{k\ge0}U_{k}(x)t^{k}=\frac{1}{1-2xt+t^{2}}.
\end{equation}
Choosing $x=\frac{p}{2}$ with $p$ integer, it follows that
\begin{equation}
\sum_{m\ge0}U_{m}\left(\frac{p}{2}\right)t^{m}=\frac{1}{1-pt+t^{2}}=\frac{1}{1-t(p-t)}=\sum_{k\ge0}t^{k}(p-t)^{k}
\end{equation}
since by choosing $t$ small enough, $\vert t(p-t)\vert<1$. The
coefficient of $t^{m}$ in this sum, which is $U_{m}\left(\frac{p}{2}\right)$, is
clearly an integer.
\end{proof}

\begin{lemma}
\label{lem:den4}
The denominator of $B^{*}_{n}(j)$ is
independent of $j\in\ZZ$. In other words,
for all $j\in\ZZ$,
\begin{equation}
  \op{denom}B^*_n(j) = \op{denom}B^*_n.
\end{equation}
\end{lemma}
\begin{proof}
Assume $j > 0$. It is a consequence of Theorem \ref{thm:2adic} that the 
denominator of $B_{n}^{*}$
 is divisible by $4$, and thus is $4t$ for some $t\in\ZZ$.

Assume, therefore, by induction that the denominator of $B_{n}^{*}(j)$ is $4t$ as well; that is, in reduced form
\begin{equation}
B_{n}^{*}(j) = \frac{x}{4t},
\end{equation}
\noindent
with $x = x(j)$ an odd integer. The
identity \eqref{cheby-zag0} coupled with Lemma \ref{lem:den4} gives
\begin{equation}
B_{n}^{*}(j+1) = \frac{x}{4t} + \frac{w}{2} = \frac{x+2wt}{4t},
\label{new-sum}
\end{equation}
\noindent
with $w \in \mathbb{Z}$. The last fraction in \eqref{new-sum} is also in
reduced form. Indeed, the numerator is odd so there is no cancellation of the
factor $4$ and if $p$ is an odd prime that divides both $x+2wt$ and $4t$, then
it divides $\gcd(x,t) = 1$. Therefore $B_{n}^{*}(j+1)$ also has denominator
$4t$, the denominator of $B_{n}^{*}$. This proof easily adapts to the case when $j$ is negative.
\end{proof}

\section{An asymptotic expansion related to the numbers $B^*_n$}

The generating function
\begin{align*}
  \sum_{n=1}^\infty B^*_n(x) z^n = -\frac12 \log z - \frac12 \psi
\left( z+1/z-1-x \right)
\end{align*}
appears as Theorem 5.1 of \cite{dixit-2012a}. Here $\psi(z)$ is the
digamma function
\begin{equation}
\psi(z) = \frac{\Gamma'(z)}{\Gamma(z)},
\end{equation}
the logarithmic derivative of the gamma function. The asymptotic
expansion for the auxiliary function
\begin{equation}
V(z) = \log z + \psi \left( z + \frac{1}{z} \right)
\label{V-def}
\end{equation}
\noindent
as $z \to 0$ in the form
\begin{equation}
V(z) \sim \sum_{n=0}^{\infty} v_{n} z^{n}
\label{V-asym}
\end{equation}
will yield a relation between the numbers $B^*_n$ and the sequence
$v_n$ in \eqref{V-asym}.

%The next lemma establishes a preliminary result used to prove that
%$\alpha_{2n+1} = 4$.
%
%\begin{lemma}
%\label{shift-psi}
%For $j \in \mathbb{N}$
%\begin{equation}
%\psi \left( z + \frac{1}{z} + 2 + j \right) =
%\psi \left( z + \frac{1}{z} \right) + \sum_{r=0}^{j+1} \frac{1}{r + z + 1/z}
%\end{equation}
%\noindent
%and
%\begin{equation}
%\psi \left( z + \frac{1}{z} -1 -  j \right) =
%\psi \left( z + \frac{1}{z} \right) + \sum_{r=1}^{j+1} \frac{1}{r - z - 1/z}.
%\end{equation}
%\end{lemma}
%\begin{proof}
%Iterate the functional equation
%\begin{equation}
%\psi(u+1) = \psi(u) + \frac{1}{u}
%\end{equation}
%\noindent
%to obtain
%\begin{equation}
%\psi(u+k) =  \psi(u) + \sum_{r=0}^{k-1} \frac{1}{u+r}.
%\label{itera-psi}
%\end{equation}
%\noindent
%This gives the result.
%\end{proof}

The value of $\alpha_{2n+1}$ has been established in \cite{dixit-2012a}.

\begin{theorem}
\label{thm-alphaodd}
For $j \in \mathbb{Z}$, the coefficients $4B_{2n+1}^{*}(j)$ are odd
integers. This gives
\begin{equation}
\alpha_{2n+1} = 4.
\end{equation}
\end{theorem}

%\begin{proof}
%The generating  function
%\begin{equation}
%\label{gen-b2np1}
%\sum_{n=1}^{\infty} B_{2n+1}^{*}(j)z^{2n+1} =
%- \frac{1}{4} \psi \left( z + \frac{1}{z} + 2 + j \right)
%- \frac{1}{4} \psi \left( z + \frac{1}{z} -1 - j \right),
%\end{equation}
%\noindent
%established as Corollary $5.3$ in \cite{dixit-2012a}
%and Lemma \ref{shift-psi} give
%\begin{eqnarray*}
%\sum_{n=0}^{\infty} 4 B_{2n+1}^{*}(j)z^{2n+1} & = &
%\psi \left( z + \frac{1}{z} + j + 2 \right) -
%\psi \left( z + \frac{1}{z} - j - 1 \right)  \\
%& = & \sum_{k=0}^{j+1} \frac{1}{z + 1/z + k} +
%\sum_{k=0}^{j+1} \frac{1}{z + 1/z -j-1 + k}  - \frac{1}{z+ 1/z}.
%\end{eqnarray*}
%\noindent
%Replace $k$ by $j+1-k$ in the second sum to obtain
%\begin{eqnarray*}
%\sum_{n=0}^{\infty} 4 B_{2n+1}^{*}(j)z^{2n+1} & = &
%\sum_{k=0}^{j+1} \left( \frac{1}{z + 1/z + k} +
%\frac{1}{z + 1/z -k} \right) - \frac{1}{z+ 1/z} \\
%& = & 2z \sum_{k=0}^{j+1} \frac{z^{2}+1}{(z^{2}+1)^{2} - k^{2}z^{2}}
%- \frac{z}{z^{2}+ 1}.
%\end{eqnarray*}
%\noindent
%The coefficients of the first term are even integers (due to the factor $2z$)
%and those of the second term are $\pm 1$. Therefore there is no further
%cancellation in this expansion. This gives the result.
%\end{proof}

The generating function for the much more involved case of $\alpha_{2n}$ is
\begin{equation*}
\sum_{n=1}^{\infty} B_{2n}^{*}(j) z^{2n}  =
- \frac{1}{2} \log z -
\frac{1}{4} \psi \left( z + \frac{1}{z} + 2 + j \right) -
\frac{1}{4} \psi \left( z + \frac{1}{z} - 1 - j \right).
\end{equation*}
This was given in Corollary $5.3$ of \cite{dixit-2012a} and can be converted to
\begin{eqnarray*}
\sum_{n=1}^{\infty} B_{2n}^{*}(j) z^{2n} & = &
- \frac{1}{2} \log z -
\frac{1}{2} \psi \left( z + \frac{1}{z} \right)  \\
& & - \frac{1}{4} \sum_{r=0}^{j+1}
\left[ \frac{z}{z^{2} + rz + 1 } + \frac{z}{z^{2}-rz+1} \right] +
\frac{z}{4(z^{2}+1)}
\end{eqnarray*}
\noindent
using 
\begin{equation}
\psi(u+k) =  \psi(u) + \sum_{r=0}^{k-1} \frac{1}{u+r}.
\label{itera-psi}
\end{equation}
\noindent
Now use the function $V(z)$ defined in
\eqref{V-def} to obtain
\begin{equation}\label{eq:genfunBseven}
\sum_{n=1}^{\infty} B_{2n}^{*}(0)z^{2n} = - \frac{1}{2} V(z) -
\frac{z}{4} \left[ \frac{1}{z^{2}+1} + \frac{2(1-z^{4})}{1-z^{6}} \right].
\end{equation}

This identity shows that Question \ref{q:main} is indeed equivalent to the
rational numbers $v_{2n}$ having even denominators.

A direct symbolic computation gives the values of the first few $v_{n}$ as
\begin{equation}
\left\{ 0, \, - \frac{1}{2}, \, \frac{11}{12}, \, \frac{1}{2}, \,
- \frac{13}{40}, \, - \frac{1}{2}, \, \frac{29}{630}, \, \frac{1}{2},
\, \frac{109}{560}, \, -\frac{1}{2}, - \frac{67}{132}, \, \frac{1}{2},
\, \frac{6571}{6006}  \right\}.
\end{equation}
\noindent
This data suggests that $|v_{n}| = 1/2$ for $n$ odd but no simple pattern is
observed for $n$ even.

\section{The use of bounds on $\psi(z)$}
\label{sec-bounds}

The first approach to the computation of the
coefficients $v_{n}$ is to use bounds for the digamma function
$\psi(z)$ and its derivatives that exist in the literature. This process
succeeds only for small values of $n$.

\begin{proposition}
\label{value-a0}
The function $V(z)$ satisfies
\begin{equation}
\lim\limits_{z \to 0^{+}} V(z) = 0,
\end{equation}
\noindent
that is, $v_{0} = 0$.
\end{proposition}
\begin{proof}
The inequality
\begin{equation}
\frac{1}{2z} < \log z - \psi(z) < \frac{1}{z}
\label{ineq-alzer1}
\end{equation}
\noindent
was established by H. Alzer \cite{alzer-1997a}. This gives
\begin{equation}
\log(z^{2}+1) - \frac{z}{z^{2}+1} < V(z) < \log(z^{2}+1) - \frac{z}{2(z^{2}+1)}
\end{equation}
\noindent
and the result follows from here. The inequality \eqref{ineq-alzer1} has been
improved by F. Qi and B. Guo \cite{guo-2010a} to
\begin{equation}
\log \left( z + \frac{1}{2} \right)  - \frac{1}{z} < \psi(z) <
\log \left( z + e^{-\gamma} \right) - \frac{1}{z}.
\end{equation}
\end{proof}

The next statement shows the computation of $v_{1}$. It requires sharper
bounds on the derivative $\psi'(x)$. The proof presented below should be seen
as a sign that a different procedure is desirable for the evaluation of
general $v_{n}$.

\begin{proposition}
\label{value-a1}
The function $V(z)$ satisfies
\begin{equation}
\lim\limits_{z \to 0^{+}} V'(z) =  - \frac{1}{2},
\end{equation}
\noindent
that is, $v_{1} = -1/2$.
\end{proposition}
\begin{proof}
The inequalities
\begin{equation}
\frac{(k-1)!}{z^{k}} + \frac{k!}{2z^{k+1}} < (-1)^{k+1}\psi^{(k)}(z) <
\frac{(k-1)!}{z^{k}} + \frac{k!}{z^{k+1}}, \quad \text{ for }z > 0,
\end{equation}
\noindent
are established in \cite{guo-2006a}. In the special case $k=1$ they produce
\begin{equation}
\frac{1}{z} + \frac{1}{2z^{2}} < \psi'(z) < \frac{1}{z} + \frac{1}{z^{2}}.
\end{equation}
\noindent
It turns out that the lower bound gives a sharp result for $V'(z)$ as
$z \to 0^{+}$. Indeed,
\begin{equation}
V'(z) = \left( 1 - \frac{1}{z^{2}} \right) \psi' \left( z + \frac{1}{z} \right)
+ \frac{1}{z} < \frac{4z^{3} + z^{2} + 4z-1}{2(1+z^{2})^{2}}.
\label{ineq-1}
\end{equation}
\noindent
The reader should check that the upper bound does not give useful information.
Instead the inequality
\begin{equation}
\psi'(z) < e^{1/z} - 1,
\end{equation}
\noindent
established in \cite{guo-2009a}, is used to produce
\begin{equation}
V'(z) > \left( \frac{z^{2}-1}{z^{2}} \right)
\left[ \text{exp} \left( \frac{z}{z^{2}+1} \right) -1 \right] + \frac{1}{z}.
\label{ineq-2}
\end{equation}
\noindent
The result now follows by letting $z \to 0$ in \eqref{ineq-1} and
\eqref{ineq-2}.
\end{proof}

The computation of $v_{n}$ by this procedure requires bounds on
all derivatives of $\psi(x)$. The examples discussed above shows that this is
not an efficient procedure. The next section presents an alternative.

\section{The computation of $v_{n}$ by umbral calculus}
\label{sec-auxiliary}

The goal of this section is to compute the coefficients $v_{n}$ in
the expansion \eqref{V-asym} by the
techniques of umbral calculus. The reader is referred to
\cite{dixit-2012a} for an introduction to these techniques and for the
statements used in this section.

Introduce the auxiliary function
\begin{equation}
F(x) = \psi \left( \frac{1}{x} \right) + \log x,
\end{equation}
\noindent
for $x> 0$ and observe that
\begin{equation}
V(z) = F \left( \frac{z}{z^{2}+1} \right) + \log(z^{2}+1).
\end{equation}

\begin{theorem}
\label{asym-F}
The function $F(x)$ admits the asymptotic expansion
\begin{equation}
F(x) \sim  \sum_{n=1}^{\infty} \frac{(-1)^{n+1}B_{n}}{n} x^{n}.
\label{series-F}
\end{equation}
\end{theorem}
\begin{proof}
The integral representation
\begin{equation}
\psi(z) = \log z + \int_{0}^{\infty} e^{-tz}
\left( \frac{1}{t} - \frac{1}{1-e^{-t}} \right) \, dt
\label{int-rep1}
\end{equation}
\noindent
produces
\begin{equation}
F(x) =  \int_{0}^{\infty} e^{-t/x}
\left( \frac{1}{t} - \frac{1}{1-e^{-t}} \right) \, dt.
\label{int-rep2}
\end{equation}
\noindent
Set $s = t/x$ to obtain
\begin{equation}
F(x) =  \int_{0}^{\infty} \frac{e^{-s}}{s}
\left( 1 - \frac{sxe^{sx}}{e^{sx}-1} \right) \, ds.
\label{int-rep3}
\end{equation}
\noindent
The generating function for the Bernoulli polynomials
\begin{equation}
\frac{te^{xt}}{e^{t}-1} = \sum_{n=0}^{\infty} B_{n}(x) \frac{t^{n}}{n!}
\end{equation}
\noindent
yields
\begin{eqnarray*}
F(x) & = &   \int_{0}^{\infty} \frac{e^{-s}}{s}
\left( 1 - \sum_{n=0}^{\infty} \frac{B_{n}(1) (sx)^{n}}{n!}  \right) \, ds \\
& = &   -\int_{0}^{\infty} \frac{e^{-s}}{s}
\sum_{n=1}^{\infty} \frac{B_{n}(1) (sx)^{n}}{n!} \, ds \\
& = &   -\sum_{n=1}^{\infty} \frac{B_{n}(1)x^{n}}{n!}
\int_{0}^{\infty} e^{-s}s^{n-1} \, ds.
\end{eqnarray*}
\noindent
The result now follows from $B_{n}(1) = (-1)^{n}B_{n}$.
\end{proof}

\begin{note}
The asymptotic behavior
\begin{equation}
|B_{2n} | \sim 4 \sqrt{\pi n} \left( \frac{n}{\pi e } \right)^{2n}
\end{equation}
\noindent
shows that the series in \eqref{series-F} does not converge for $x \neq 0$.
\end{note}

The result in Theorem \ref{asym-F} is now transformed using the umbral method
described in \cite{dixit-2012a}. 
The essential point is the introduction of an umbra
$\mathfrak{B}$ for the Bernoulli polynomials  $B_{n}(x)$ by
the generating function
\begin{equation}
\text{eval} \left\{ \text{exp}(t \mathfrak{B}(x)) \right\} =
\frac{te^{xt}}{e^{t}-1}
\end{equation}
\noindent
The rules
$\text{eval} \left(\mathfrak{B}^{n} \right) =  B_{n}$
and
$\text{eval} \left(\mathfrak{B}(x) \right) =  \text{eval} \left\{ x +
\mathfrak{B} \right\}$ are
useful in converting identities  involving Bernoulli polynomials.

\begin{theorem}
The coefficients $v_{n}$ in the expansion  \eqref{V-asym}
are given by
\begin{equation}
v_{n} = \sum_{k=0}^{\lf n/2 \rf} (-1)^{n-k+1}
\frac{\binom{n-k}{k}}{n-k} B_{n-2k}.
\label{form1-vn}
\end{equation}
\end{theorem}
\begin{proof}
The result of Theorem \ref{asym-F} can be written as
\begin{eqnarray}
F(x) & = & \sum_{n=1}^{\infty} \frac{(-1)^{n+1}}{n} (x \mathfrak{B})^{n} \\
 & = & \log( 1 + x \mathfrak{B} ).  \nonumber
\end{eqnarray}
Then
\begin{eqnarray*}
V(x) & = & F \left( \frac{x}{x^{2}+1} \right) + \log(x^{2}+1) \\
 & = & \text{eval} \left(
\log \left( 1 + \frac{x \mathfrak{B}}{x^{2}+1} \right) +
\log(x^{2}+1) \right) \\
& = & \text{eval} \left( \log \left( x^{2} + 1 + x \mathfrak{B} \right)
\right) \\
& = & \text{eval} \left(
\sum_{r=1}^{\infty} \frac{(-1)^{r+1}}{r} x^{r} ( x + \mathfrak{B})^{r}
\right) \\
& = & \sum_{r=1}^{\infty} \frac{(-1)^{r+1}}{r} x^{r} B_{r}(x) \\
& = & \sum_{r=1}^{\infty} \frac{(-1)^{r+1} x^{r}}{r}
\sum_{k=0}^{r} \binom{r}{k} B_{r-k}x^{k}.
\end{eqnarray*}
\noindent
Now let $n = r+k$ and invert the order of summation to obtain the result.
\end{proof}

Separating the expression for the coefficients $v_{n}$ given in
\eqref{form1-vn} according to the parity of $n$, simplifies the result.

\begin{corollary}
\label{thm-vnform}
The coefficients $v_{n}$ in \eqref{V-asym} are given by
\begin{eqnarray}
v_{2n-1} & = & \frac{(-1)^{n}}{2}, \label{form1-vna0}  \\
& &  \nonumber \\
v_{2n} & = & (-1)^{n+1} \left[
\frac{1}{n} + \sum_{k=1}^{n} (-1)^{k}
\binom{n+k-1}{n-k} \frac{B_{2k}}{2k} \right]. \label{form1-vna1}
\end{eqnarray}
\end{corollary}

\section{Properties of Zagier polynomials give the expression for $v_{n}$}
\label{sec-zagier}

This section presents a proof of the expressions for $v_{n}$ given in
Corollary \ref{thm-vnform} by using properties of the Zagier polynomials
established in \cite{dixit-2012a}.

Theorem $5.1$ in \cite{dixit-2012a} gives the generating function of the
Zagier polynomials
\begin{equation}
\sum_{n=1}^{\infty} B_{n}^{*}(x)z^{n} = - \frac{\log z}{2} -
\frac{1}{2} \psi \left( z + \frac{1}{z} - 1 - x \right)
\end{equation}
\noindent
that for $x=-1$ yields
\begin{equation}
\sum_{n=1}^{\infty} B_{n}^{*}(-1)z^{n} = - \frac{\log z}{2} -
\frac{1}{2} \psi \left( z + \frac{1}{z}  \right).
\end{equation}
\noindent
Comparing with the asymptotics for $V(z)$ given in
\eqref{V-asym} gives the next statement.

\begin{proposition}
The coefficients $v_{n}$ are given by
\begin{equation}
v_{n} = - 2 B_{n}^{*}(-1).
\label{value-vn}
\end{equation}
\end{proposition}

To obtain an expression for
$B_{n}^{*}(-1)$ use \eqref{cheby-zag0} 
with $n$ replaced by
$2n+1$ and $x=-1$. It follows that
\begin{equation}
B_{2n+1}^{*}(-1) = B_{2n+1}^{*}(0) - \frac{1}{2} U_{2n} \left( \frac{1}{2}
\right).
\label{form-000}
\end{equation}
\noindent
The reduction of this expression uses
Theorem $10.1$ in \cite{dixit-2012a} in the form
\begin{equation}
2B_{2n+1}^{*}(x) = \sum_{r=0}^{n} (-1)^{n+r} \binom{n+r+1}{2r+1}
\frac{B_{2r+1}(x)}{n+r+1} + U_{2n} \left( \frac{x}{2} \right) +
U_{2n} \left( \frac{x+1}{2} \right),
\end{equation}
\noindent
which in the special case $x=0$ produces
\begin{equation}
B_{2n+1}^{*}(0) = \frac{(-1)^{n}}{4} + \frac{1}{2} U_{2n} \left( \frac{1}{2}
\right)
\end{equation}
\noindent
using $U_{2n}(0) = (-1)^n$. Inserting 
in \eqref{form-000} gives the result for odd index.

In the case of even index, the proof starts with the reflection symmetry
of the Zagier polynomials
\begin{equation}
B_{n}^{*}(-x-3) = (-1)^{n} B_{n}^{*}(x)
\end{equation}
\noindent
(given as Theorem $11.1$ in \cite{dixit-2012a}) which in
the special case $x=-2$ gives
\begin{equation}
-2B_{2n}^{*}(-1) = -2 B_{2n}^{*}(-2).
\label{v2n-0}
\end{equation}
\noindent
To obtain the expression for $v_{2n}$, use the
identity $(10.10)$ in \cite{dixit-2012a}
\begin{equation}
\sum_{r=0}^{n} (-1)^{n+r} \binom{n+r}{2r} \frac{B_{2r}(x)}{n+r} =
2 B_{2n}^{*}(x-2)
\label{bstar-59}
\end{equation}
\noindent
in the special case $x=0$.
 This gives the values of $v_{n}$ stated in
Corollary \ref{thm-vnform}. Thus \eqref{form1-vna1} and \eqref{v2n-0} imply
\eqref{value-vn}.

\section{Calculation of $v_{n}$ by an asymptotic method}
\label{sec-glasser}

The goal of this section is to derive the formula for $v_{n}$ by a
direct asymptotic expansion of the digamma function:
\begin{equation}
\psi(z) \sim \log z - \frac{1}{2z} - 
\sum_{k=1}^{\infty} \frac{B_{2k}}{2kz^{2k}},
\text{ as }
z \to \infty.
\label{asym-psi5}
\end{equation}

Start with
\begin{equation}
V(z) = \psi \left( \frac{z^{2}+1}{z} \right) - \log \left( \frac{z^{2}+1}{z}
\right) + \log(z^{2}+1)
\end{equation}
\noindent
and use \eqref{asym-psi5} to obtain
\begin{eqnarray*}
V(z) & \sim & - \sum_{k=1}^{\infty} \frac{B_{k}}{k}
\left( \frac{z}{z^{2}+1} \right)^{k} + \log(z^{2}+1) \\
& = & \frac{z}{2(z^{2}+1)} - \sum_{k=1}^{\infty} \frac{B_{2k}}{2k}
\left( \frac{z}{z^{2}+1} \right)^{2k} + \log(z^{2}+1) \\
& = & \sum_{n=1}^{\infty} \frac{(-1)^{n+1}}{n} z^{2n} -
\frac{1}{2} \sum_{n=0}^{\infty} (-1)^{n} z^{2n+1} -
\sum_{k=1}^{\infty} \frac{B_{2k}}{2k} z^{2k}
\sum_{\ell =0}^{\infty} \frac{(-1)^{\ell} (\ell - 1 +2k)!}{\ell! (2k-1)! }
z^{2 \ell}.
\end{eqnarray*}
\noindent
The coefficient of the odd powers of $z$ can be read immediately. Indeed,
\begin{equation}
v_{2n-1} = \frac{(-1)^{n}}{2}.
\end{equation}
\noindent
This is \eqref{form1-vna0}. To obtain the
expression for the even powers, observe that
\begin{equation*}
\sum_{k=1}^{\infty} \frac{B_{2k}}{2k} z^{2k}
\sum_{\ell =0}^{\infty} \frac{(-1)^{\ell} (\ell - 1 +2k)!}{\ell! (2k-1)! }
z^{2 \ell} =
\sum_{i=1}^{\infty} (-1)^i
\left( \sum_{k=1}^{i} (-1)^{k} \binom{i+k-1}{2k-1}
\frac{B_{2k}}{2k} \right) z^{2i}.
\end{equation*}
\noindent
This gives
\begin{equation}
v_{2n} = \frac{(-1)^{n}}{n} +
\sum_{k=1}^{n} (-1)^{k} \binom{n+k-1}{2k-1} \frac{B_{2k}}{2k}.
\end{equation}
\noindent
This is equivalent to \eqref{form1-vna1} and also to \eqref{bstar-59} with
$x=0$.

\medskip

An expression for $v_{2n}$ in terms of Chebyshev polynomials in given next.

\begin{proposition}
Let $T_{n}(x)$ be the Chebyshev polynomial of the first kind. Then
\begin{equation}
v_{2n} = - \text{eval} \left(
\frac{1}{n} T_{2n} \left( \frac{\mathfrak{B}}{2} \right) \right) =
- \frac{1}{n} T_{n} \left( \frac{\mathfrak{B}^{2}-2}{2} \right).
\end{equation}
\end{proposition}
\begin{proof}
Lemma $9.2$ in \cite{dixit-2012a} established  the representation
\begin{equation}
B_{n}^{*}(x) = \text{eval} \left(
\frac{1}{n} T_{n} \left( \frac{\mathfrak{B} + x + 2}{2} \right) \right).
\end{equation}
\noindent
The relation \eqref{v2n-0} gives $v_{2n} = - 2B_{2n}^{*}(-2)$ so that
\begin{equation}
v_{2n} = - \text{eval} \left(
\frac{1}{n} T_{2n} \left( \frac{\mathfrak{B}}{2} \right) \right).
\end{equation}
\noindent
The result now follows from the identity $T_{2n}(x) = T_{n}(2x^{2}-1)$ for
Chebyshev polynomials; see \cite{brychkov-2008a}, $7.210$ formula 
7 on page $550$.
\end{proof}

\section{The asymptotics of $\psi(z)$ and its derivatives}
\label{sec-asymptotics}

The coefficients $v_{n}$ in the expansion \eqref{V-asym} are now evaluated
from the expression
\begin{equation}
v_{n} = \frac{1}{n!} \lim\limits_{z \to 0} \left( \frac{d}{dz} \right)^{n}
V(z).
\label{vn-zero}
\end{equation}
%The first application of Proposition \ref{asym-psi-final} is a second
%proof of the fact that $v_{0}=0$, given first in Proposition \ref{value-a0}.
%
%\begin{example}
%The identity
%\begin{equation}
%V(z) = \psi \left(z + \frac{1}{z} \right) - \log \left( z + \frac{1}{z} \right)
%+ \log(z^{2}+1)
%\end{equation}
%\noindent
%shows that
%\begin{equation}
%\lim\limits_{z \to 0^{+}} V(z) =
%\lim\limits_{t \to \infty} \psi(t) - \log t.
%\end{equation}
%\noindent
%This last limit is zero in view of the asymptotic expansion \eqref{asym-psi5}.
%It follows that $v_{0}=0$.
%\end{example}
%

The next theorem shows
existence of a  sequence of polynomials $A_{j,n}(z)$ that give the desired
formula for derivatives of $V(z)$. Theorem \ref{thm-koutschan} presented below
provides an explicit form of these polynomials.

\begin{theorem}
\label{thm-polyV}
Let $n \in \mathbb{N}$. Then there are polynomials $A_{j,n}(z)$, with
$1 \leq j \leq n$ such that
\begin{equation}
z^{2n} \left( \frac{d}{dz} \right)^{n} V(z) =
(-1)^{n-1} (n-1)! z^{n} + \sum_{j=1}^{n} A_{j,n}(z) \psi_{j}(z + 1/z).
\label{polynomial-V}
\end{equation}
\noindent
The polynomials $A_{j,n}(z)$ satisfy the recurrences
\begin{eqnarray*}
A_{n+1,n+1}(z) & = & (z^{2}-1) A_{n,n}(z), \nonumber  \\
A_{j,n+1}(z) & = & -2nzA_{j,n}(z) + z^{2}A_{j,n}'(z) + (z^{2}-1)A_{j-1,n}(z)
\quad \text{ for } 2 \leq j \leq n, \nonumber \\
A_{1,n+1}(z) & = & -2nz A_{1,n}(z) + z^{2}A_{1,n}'(z), \nonumber
\end{eqnarray*}
\noindent
and the initial condition
\begin{equation*}
A_{1,1}(z) = z^{2}-1.
\end{equation*}
\noindent
The degree of $A_{j,n}(z)$ is $n+j-2$ if $1 \leq j \leq n-1$
and $2n$ for $j=n$.
\end{theorem}
\begin{proof}
The term $(-1)^{n-1}(n-1)!z^{-n}$ arises from the $n$-th derivative of
$\log z$. To obtain the recurrences, simply observe that
\begin{eqnarray*}
\left( \frac{d}{dz} \right)^{n+1} \psi \left( z + 1/z \right)
& = &
\left( \frac{d}{dz} \right)
\left[ z^{-2n} \sum_{j=1}^{n} A_{j,n}(z) \psi_{j} \left( z + 1/z \right)
\right]
\end{eqnarray*}
\noindent
and compare the coefficients of $\psi_{j}(z+1/z)$. The statement about the
degree of $A_{j,n}(z)$ is obtained directly from the recurrence.
\end{proof}

The next theorem gives an explicit form of the polynomials $A_{j,n}(z)$. The
authors wish to thank C. Koutschan who used his symbolic package to
solve the recurrences in Theorem \ref{thm-polyV}.

\begin{theorem}
\label{thm-koutschan}
The polynomials $A_{j,n}(z)$ are given by
\begin{equation}
A_{n,n}(z) = (z^{2}-1)^{n}
\end{equation}
\noindent
and for $1 \leq j < n$,
\begin{equation}
A_{j,n}(z) = (-1)^{n} \frac{n!}{j!} z^{n-j} \sum_{r=0}^{j-1}
(-1)^{r} \binom{n-1-r}{n-j} \binom{j}{r} z^{2r}.
\end{equation}
\end{theorem}
\begin{proof}
Simply check that the form stated in this theorem satisfies the
recurrence given in Theorem \ref{thm-polyV}.
\end{proof}

\begin{note}
The package of C. Koutschan actually gives the form
\begin{equation}
A_{j,n}(z) = (-1)^{n} z^{n-j} \binom{n-j}{j-1} \binom{n}{j} (n-j)!
\, {_{2}F_{1}}\left( 1-j, -j ; 1-n; z^{2} \right).
\end{equation}
\noindent
The hypergeometric representation of the Jacobi polynomials
\begin{equation}
P_{m}^{(\alpha,\beta)}(x) = \binom{m+\alpha}{m}
\, {_{2}F_{1}}\left( -m, m+ \alpha + \beta + 1 ; \alpha + 1;
\frac{1-x}{2} \right)
\end{equation}
\noindent
shows that
\begin{equation}
A_{j,n}(z) = (-1)^{n+j-1} \frac{n!}{j} z^{n-j}
P_{j-1}^{(-n, n-2j)}(1-2z^{2}).
\end{equation}
\end{note}

\begin{note}
The coefficients $v_{n}$ are now obtained from \eqref{vn-zero} and the
identity
\begin{equation}
\left( \frac{d}{dz} \right)^{n} V(z) =
(-1)^{n-1} (n-1)! z^{-n} + z^{-2n} \sum_{j=1}^{n} A_{j,n}(z) \psi_{j}(z + 1/z).
\label{polynomial-V1}
\end{equation}
\noindent
This employs the expansion
\begin{equation}
\psi^{(j)}(z) = \psi_{j}(z) \sim (-1)^{j-1}
\left[ \frac{(j-1)!}{z^{j}} + \frac{j!}{2z^{j+1}}
+ \sum_{k=1}^{\infty} B_{2k} \frac{(2k+j-1)!}{(2k)! \, z^{2k+j}} \right],
\label{asym-psi1}
\end{equation}
(that appears as $6.4.11$ in \cite{abramowitz-1972a}).
The polygamma function, which appear differentiating \eqref{V-def} to 
obtain \eqref{vn-zero}, has argument $z+1/z$. Thus
\eqref{asym-psi1} is used in the form
\begin{eqnarray*}
\psi_{j}(z+1/z) & \sim &
(-1)^{j-1}
\left[ \frac{(j-1)!z^{j}}{(z^{2}+1)^{j}} +
\frac{j!z^{j+1}}{2(z^{2}+1)^{j+1}} +
\sum_{k=1}^{\infty} \frac{B_{2k} (2k+j-1)! z^{2k+j}}{(2k)! (z^{2}+1)^{2k+j}}
\right] \\
& = & \frac{(-1)^{j-1} z^{j}}{(z^{2}+1)^{j}}
\sum_{k=0}^{\infty} \frac{(-1)^{k} (k+j-1)!}{k!}
\frac{B_{k}z^{k}}{(z^{2}+1)^{k}}
\end{eqnarray*}
\noindent
as $z \to 0$.
\end{note}

\begin{proposition}
\label{asym-psi-final}
The asymptotic  expansion
\begin{eqnarray*}
\psi_{j}(z+ 1/z) & \sim &
\frac{(-1)^{j-1}}{2} z^{j+1} \sum_{r=0}^{\infty} \frac{(-1)^{r} (j+r)!}
{r!} z^{2r}   \\
&  &  + (-1)^{j-1}z^{j}
\sum_{\ell=0}^{\infty}
\left[ \sum_{k=0}^{\ell}
(-1)^{\ell - k} B_{2k} \frac{(k+j+\ell - 1)!}{(2k)! \, (\ell - k)!}
\right] \, z^{2 \ell}
\end{eqnarray*}
\noindent
holds as $z \to 0$.
\end{proposition}

A direct non-illuminating computation of the expansion in
\eqref{polynomial-V1} gives the values of $v_{n}$ in
Theorem \ref{thm-vnform}. Given the fact that other proofs of this result
have been provided, the long but elementary details are omitted.

\section{Calculation of $v_{n}$ via integral representations and the Fa\`{a} di Bruno formula}
\label{sec-int-rep}
\setcounter{equation}{0}

This section employs the integral representation
\begin{equation} \label{int-rep}
\psi(x) = \log x + \int_{0}^{\infty} \left( \frac{1}{t} - \frac{1}{1-e^{-t}}
\right) \, e^{-tx} \, dt,
\end{equation}
\noindent
of the digamma function, given as entry $8.361.8$ in \cite{gradshteyn-2007a},
to obtain the values of $v_{n}$ given in Corollary \ref{thm-vnform}.

\begin{lemma}
The function $V(z)$ in \eqref{V-def} is  expressed as
\begin{equation}
V(z) = \log(z^{2}+1) + \int_{0}^{\infty} \left( \frac{1}{t} - \frac{1}{1-e^{-t}} \right) e^{-t(z+1/z)} \, dt.
\label{eq-forV}
\end{equation}
\end{lemma}

The representation \eqref{eq-forV} reduces the computation of $v_{n}$ to
the asymptotic expansion of
\begin{equation}
W(z) = \int_{0}^{\infty} \left( \frac{1}{t} - \frac{1}{1-e^{-t}} \right) e^{-t(z+1/z)} \, dt. \label{W}
\end{equation}
\noindent
Indeed, if
\begin{equation}
V(z) \sim \sum_{n=0}^{\infty} v_{n}z^{n} \text{ and }
W(z) \sim \sum_{n=0}^{\infty} w_{n}z^{n},
\end{equation}
\noindent
then $v_{2n-1} = w_{2n-1}$ and $v_{2n} = w_{2n} + (-1)^{n-1}/n$.
The next lemma is preliminary to the computation of this expansion.

\begin{lemma}
\label{bin-ser}
\begin{equation*}
\left(\frac{z}{z^2+1}\right)^{2n} = \sum_{k=n}^\infty (-1)^{k-n}\binom{n+k-1}{k-n} z^{2k}
\end{equation*}
\end{lemma}
\begin{proof}
Use the binomial series
\[(z^2+1)^{-2n}=\sum_{i=0}^\infty \binom{-2n}{i}z^{2i}\]
to find
\[
\left(\frac{z}{z^2+1}\right)^{2n} = \sum_{i=0}^\infty \binom{-2n}{i}z^{2n+2i}=  \sum_{k=n}^\infty \binom{-2n}{k-n}z^{2k}.
\]
Now use the elementary identity
\[\binom{-2n}{k-n}=(-1)^{k-n} \binom{n+k-1}{k-n}\]
to obtain the result.
\end{proof}

To find the asymptotic expansion of the function $W(z)$ defined in 
\eqref{W}, let $s=z/(z^2+1)$, and use the change of variable $x=t/s$ to get
\begin{eqnarray*}
W(z) &=&\int_0^\infty \frac{1}{x} \left(1-\frac{xse^{xs}}{e^{xs}-1}\right)e^{-x}dx \\
&=&\int_0^\infty \frac{1}{x} 
\left(1-\sum_{n=0}^\infty \frac{B_n(1)}{n!}\left(xs\right)^n \right)e^{-x}dx \\
&=&-\int_0^\infty \frac{1}{x} 
\sum_{n=1}^\infty \frac{B_n(1)}{n!}\left(xs\right)^n e^{-x}dx.
\end{eqnarray*}

The infinite series is not uniformly convergent as $z\rightarrow 0$, and 
interchanging the sum with the integral does not provide a convergent 
series. But the resulting series
(with radius of convergence zero) will be the asymptotic expansion of $W(z)$:
\begin{eqnarray*}
W(z) &\sim&-\sum_{n=1}^\infty \frac{B_n(1)}{n!} s^n\int_0^\infty 
x^{n-1}   e^{-x}dx \\
&=& -\sum_{n=1}^\infty \frac{B_n(1)}{n!} s^n (n-1)! 
=  -\sum_{n=1}^\infty \frac{B_n(1)}{n} \left(\frac{z}{z^2+1}\right)^n\\
&=&-\frac{z}{2(z^2+1)}-\sum_{n=1}^\infty \frac{B_{2n}(1)}{2n} \left(\frac{z}{z^2+1}\right)^{2n}.
\end{eqnarray*}
The expression for the coefficients $w_n$ corresponding 
to \eqref{form1-vna1} now follows from Lemma \ref{bin-ser}.

An alternative approach based on the integral representation \ref{int-rep} 
uses the Fa\`{a} di Bruno formula and the partial Bell polynomials. Write
\[\tilde{\psi}(x) = \psi(x) - \log x= \int_0^\infty 
\left(\frac{1}{t}-\frac{1}{1-e^{-t}}\right) e^{-xt}dt, \]
so that $W(z) = \tilde{\psi}(h(z))$ with $h(z) = z + 1/z$ and
\[\tilde{\psi}^{(k)}(x) =  \int_0^\infty (-t)^k 
\left(\frac{1}{t}-\frac{1}{1-e^{-t}}\right) e^{-xt}dt.\]
Define
\begin{equation}\label{int-def}
I_k(z) = \int_0^\infty (-t)^k \left(\frac{1}{t}-\frac{1}{1-e^{-t}}\right) 
e^{-t(z+1/z)}dt =\tilde{\psi}^{(k)}(h(z)) . \end{equation}

The  {\em partial Bell polynomial}  $B_{n,k}$ in the $n-k+1$ variables
 $x_1,\ldots , x_{n-k+1}$ is defined by
$$B_{n,k}(x_1,\ldots , x_{n-k+1})=\sum_{\sigma(n,k)}\frac{n!}{j_1!j_2!\cdots j_{n-k+1}!} \left(\frac{x_1}{1!}\right)^{j_1} \left(\frac{x_2}{2!}\right)^{j_2} \cdots \left(\frac{x_{n-k+1}}{(n-k+1)!}\right)^{j_{n-k+1}},$$
where 
the sum is over the set $\sigma(n,k)$ of all non-negative integer
sequences $j_1,j_2,\ldots,j_{n-k+1}$ such that
$$j_1+j_2+ \cdots +j_{n-k+1}=k \ \ \ \mbox{and}\ \ \
 j_1+2j_2+\cdots+(n-k+1)j_{n-k+1}=n.$$

The Fa\`a di Bruno formula for the $n$-th derivative of the composition $W=\tilde{\psi} \circ h$ is then expressed as
\begin{eqnarray}W^{(n)}(z) &=& \sum_{k=1}^n \tilde{\psi}^{(k)}(h(z)) B_{n,k}\left(h'(z),\cdots , h^{(n-k+1)}(z)\right) \notag \\
&=& \sum_{k=1}^n I_k(z) B_{n,k}\left(h'(z),\cdots , h^{(n-k+1)}(z)\right). \label{Wn}
\end{eqnarray}
The next lemma provides some results on the partial Bell polynomials. A useful reference is \cite{comtet-1974a}, page 133-137.

\begin{lemma}
\begin{equation} \label{Bel-hom}
B_{n,k} (x_1, s t^2 x_2, s t^3 x_3, s t^4 x_4, \cdots ) = s^k t^n B_{n,k}\left(\frac{x_1}{st},x_2,x_3,\cdots \right)
\end{equation}

\begin{equation} \label{Bel-red}
B_{n,k}(x_1, x_2, \ldots )=\frac{n!}{(n-k)!}\sum_{\ell=0}^k \frac{1}{\ell !}x_1^\ell B_{n-k,k-\ell}\left(\frac{x_2}{2},\frac{x_3}{3},\ldots \right)
\end{equation}
\begin{equation} \label{Bel-fac}
B_{n,k}(1!,2!,3!,\ldots)=\binom{n-1}{k-1}\frac{n!}{k!}.
\end{equation}
\end{lemma}

\begin{proof}
The proof of \eqref{Bel-hom} follows easily from the definition, noting that 
 $3j_2+4j_3+\cdots = n+k-2j_1$.
Formula \eqref{Bel-red} is entry [3.l] on \cite{comtet-1974a}, and
\eqref{Bel-fac} is entry [3.h].
\end{proof}

\begin{lemma} \label{Bel-der}
The partial Bell polynomials satisfy
\begin{equation}
B_{n,k}\left(h'(z),\cdots , h^{(n-k+1)}(z)\right)=
\frac{(-1)^nn!}{z^{n+k}}\sum_{\ell=0}^k\frac{1}{\ell !}
\binom{n-k-1}{k-\ell-1}\frac{(1-z^2)^\ell}{(k-\ell)!}.
\label{Bel-der0}
\end{equation}
\end{lemma}
\begin{proof}
Note that $h'(z)=1-z^{-2}$, and $h^{(i)}(z) = (-1)^i i! z^{-i-1}$ for $i>1$.
Hence the result easily follows from \eqref{Bel-hom} 
(with $s=-1/z, t=1/z$), \eqref{Bel-red} and \eqref{Bel-fac}.
\end{proof}

The next result expresses the integrals $I_k(z)$ defined 
in \eqref{int-def} in terms of the Hurwitz zeta function
\begin{equation}
\zeta(s,q) = \sum_{n=0}^{\infty} \frac{1}{(n+q)^{s}}.
\end{equation}

\begin{proposition}
The integral $I_k(z)$ is given by
\begin{eqnarray*}
I_k(z) & = & \frac{(-1)^k \, (k-1)! \, z^{k}}{(z^{2}+1)^{k}} +
(-1)^{k-1} k! \zeta(k+1,z+1/z) \\
& = & (-1)^{k-1}(k-1)!z^{k}
\left[
kz \sum_{m=0}^{\infty} \frac{1}{(z^{2}+mz+1)^{k+1}} -\frac{1}{(z^{2}+1)^{k}} \right].
\end{eqnarray*}
\end{proposition}
\begin{proof}
The definition of the gamma function as
\begin{equation}
\Gamma(s) = \int_{0}^{\infty} t^{s-1} e^{-t} \, dt
\end{equation}
\noindent
and the integral representation for the Hurwitz zeta function
\begin{equation} \label{Hur-int}
\zeta(s,q) = \frac{1}{\Gamma(s)} \int_{0}^{\infty} \frac{t^{s-1} \,
e^{-qt} \, dt}{1 - e^{-t}}
\end{equation}
\noindent
are used in
\begin{equation*}
I_{k}(z) = (-1)^{k} \int_{0}^{\infty} t^{k-1} e^{-t(z+1/z)} \, dt -
(-1)^{k} \int_{0}^{\infty} \frac{t^{k}}{1 - e^{-t}} e^{-t(z+1/z)} \, dt
\end{equation*}
\noindent
to obtain the result.
\end{proof}

The integrals $I_k(z)$ are now expressed in terms of the Bernoulli numbers.
The proof is similar to the one given for Lemma \ref{bin-ser}, so the 
details are omitted. 

\begin{proposition}
The identity 
\begin{equation} 
\label{Ik}
I_k(z)=(-1)^{k-1} k! \left(\frac{z}{z^2+1}\right)^{k+1}\left(\frac{1}{2} +\sum_{i=1}^\infty \frac{B_{2i}}{k+2i}\binom{k+2i}{k}\left(\frac{z}{z^2+1}\right)^{2i-1}\right)
\end{equation}
\noindent
holds.
\end{proposition}

According to \eqref{Wn}, the $n$-th derivative of $W(z)$ is obtained by 
multiplying  \eqref{Bel-der0} and \eqref{Ik} and summing over $k$.
The coefficients $v_{2n}$ are then found as 
$$v_{2n}=\frac{W^{(2n)}(0)}{(2n)!} +\frac{(-1)^{n-1}}{n}.$$

In order to find explicit formulas for $W^{(2n)}(0)$,  (\ref{Bel-der}) and (\ref{Ik}) are expanded in powers of $z$, and then the constant term 
in the sum is selected. Note that (\ref{Bel-der}) is of order $z^{-n-k}$ as $z\rightarrow 0$, while (\ref{Ik}) is of order $z^{k+1}$.
So the product is of order $z^{-(n-1)}$. Since $W^{(n)}(z)$ is bounded as $z\rightarrow 0$,  after summing over $k$ 
all coefficients of $z^i$ for $i<0$ must vanish.

The computations to derive $v_{2n}$ with this approach are trivial but lengthy, and the resulting expression (involving multiple nested sums of
binomial coefficients) is not particularly illuminating, so they are omitted. 
The vanishing of the coefficients of negative powers comparing it with 
\eqref{form1-vna1} yields a 
family of identities.

\begin{proposition}
Let 
\[
\begin{array}{l}
A(i,j,k,\ell,m,r)\\[3ex ]
=\displaystyle{(-1)^{i+j+k} \binom{k}{\ell} \binom{\ell}{r} \binom{k+2i}{k} \binom{2m-k-1}{k-\ell-1}\binom{k+i+m-r-j-1}{k+2i-1}\frac{B_{2i}}{k+2i}}.
\end{array}
\]
Then
\[
\sum_{k=1}^{2m} \sum_{\ell=0}^k \sum_{r=0}^m \sum_{i=1}^{m-r-j} A(i,j,k,\ell,m,r)=\left\{\begin{array}{lcr} 0 & \mbox{if} & j>0,\\
\displaystyle{\sum_{s=1}^m (-1)^s \binom{m+s-1}{m-s}\frac{B_{2s}}{2s}} & \mbox{if} & j=0. \end{array}\right.
\]
\end{proposition}

\section{Calculation of $v_{n}$ by Hoppe's formula}
\label{sec-faa}

%One more method to derive the coefficients $v_n$ uses Hoppe's formula for the derivatives of the composition of two functions. 
%This formula works well when useful expressions can be derived for the derivatives of powers of the inner function, as is the case here for the 
%inner function $g(z)=z/(z^2+1)$. Nonetheless, the resulting expression for the coefficients involves multiple nested sums
% (as in the previous section) and will not be explicitly derived. We outline the procedure because by comparing with (\ref{form1-vna1}), 
%  this method can be used to derive identities for sums involving binomial coefficients and the polygamma functions $\psi_r(z)=\psi^{(r)}(z)$.

The function $V(z)$ in \eqref{V-def} can be written as
\begin{equation}
V(z) = F(g(z)) + \log(z^{2}+1)
\label{new-V00}
\end{equation}
\noindent
with
\begin{equation}
F(z) = \psi \left( \frac{1}{z} \right) + \log z \text{ and }
g(z) = \frac{z}{z^{2}+1}.
\label{F-def0}
\end{equation}
\noindent
The expansion
\begin{equation}
\log(z^{2}+1) = \sum_{n=1}^{\infty} \frac{(-1)^{n-1}}{n} z^{2n}
\end{equation}
\noindent
is elementary, therefore the coefficients $v_{n}$ in the expansion
\eqref{V-asym} are now evaluated from $F(g(z))$.

Hoppe's formula for the derivative of compositions of functions is stated
in the next theorem. See \cite{johnsonw-2002a} for details. 

\begin{theorem}
\label{hoppe-thm}
Assume that all derivatives of $g$ and $F$ exist, then
\begin{equation} \label{hoppe-form}
\left( \frac{d}{dz} \right)^{n} F(g(z)) =
\sum_{k=0}^{n} \frac{P_{n,k}(g(z))}{k!}
F^{(k)}(g(z)),
\end{equation}
\noindent
where
\begin{equation}
P_{n,k}(g(z)) = \sum_{j=0}^{k} (-1)^{k-j} \binom{k}{j}
\left[ g(z) \right]^{k-j}
\left( \frac{d}{dz} \right)^{n} \left[ g(z) \right]^{j}
\end{equation}
\noindent
and $P_{n,0}(0) = 0$ for $n > 0$.
\end{theorem}

Hoppe's formula is now used to compute the $n$-th derivative of the function
$F(g(z))$, where $F$ is defined in \eqref{F-def0} and $g(z) = z/(z^{2}+1)$.
The formula requires
\begin{equation}
F^{(k)}(z)=\left( \frac{d}{dz} \right)^{k} F(z)
\text{ and }
\left( \frac{d}{dz} \right)^{n} \left[ g(z)  \right]^{j}.
\end{equation}
\noindent
These terms are computed next.

\begin{lemma}
Let $F(z) = \psi(1/z) + \log z$ and $\psi_{r}(z) = \left( \frac{d}{dz}
\right)^{r} \psi(z)$. Then, if $k\geq 1$,
\begin{equation}
F^{(k)}(z) =
\frac{(-1)^{k} k!}{z^{k}} \sum_{r=1}^{k} \frac{1}{r! z^{r}}
\binom{k-1}{r-1} \psi_{r} \left( \frac{1}{z} \right) +
\frac{(-1)^{k-1} (k-1)!}{z^{k}}.
\end{equation}
\end{lemma}
\begin{proof}
Hoppe's formula gives
\begin{equation}
\left( \frac{d}{dz} \right)^{k}
\psi \left( \frac{1}{z} \right) =
\sum_{r=0}^{k} \frac{1}{r!} P_{k,r} \left( \frac{1}{z} \right)
\times \left( \frac{d}{dz} \right)^{r}  \psi(z) \Big{|}_{z \to 1/z}
\end{equation}
\noindent
with
\begin{eqnarray}
P_{k,r} \left( \frac{1}{z} \right)  & = &
\sum_{\ell=0}^{r} (-1)^{r - \ell} \binom{r}{\ell}
\left( \frac{1}{z} \right)^{r - \ell}
\left( \frac{d}{dz} \right)^{k} \left[ \frac{1}{z^{\ell}} \right] \label{p-1} \\
& = & \frac{(-1)^{r+k}}{z^{r+k}} \sum_{l=0}^{r} (-1)^{\ell}
\frac{r! \, (\ell +k-1)!}{\ell ! (r - \ell)! (\ell - 1)!} \nonumber \\
& = & \frac{(-1)^{k}}{z^{r+k}} k! \binom{k-1}{r-1} \nonumber
\end{eqnarray}
\noindent
for $r \geq 1$ and $P_{k,0}(1/z) = 0$. The last step uses the evaluation
\begin{equation}
\sum_{\ell = 0}^{r} \frac{(-1)^{\ell} r! \, (\ell + k - 1)!}
{\ell ! \, (r - \ell)! \, (\ell - 1)!} =
(-1)^{r} k! \, \binom{k-1}{r-1}.
\end{equation}
\end{proof}

\begin{lemma}
For $g(z) = z/(z^{2}+1)$ and $n, \, j \in \mathbb{N}$:
\begin{eqnarray*}
\left( \frac{d}{dz} \right)^{n} [g(z)]^{j}  & = &
n! \sum_{r=0}^{\infty} (-1)^{r} \binom{j+r-1}{r} \binom{2r+j}{n}
z^{2r+j-n}.
\end{eqnarray*}
\end{lemma}
\begin{proof}
The binomial theorem gives
\begin{equation}
\left( \frac{z}{z^{2}+1} \right)^{j}  = z^{j} (1 + z^{2})^{-j} =
\sum_{r=0}^{\infty} (-1)^{r} \binom{j+r-1}{j-1} z^{2r+j}.
\end{equation}
\noindent
Differentiating $n$ times yields the result.
\end{proof}

The terms in Theorem \ref{hoppe-thm} are now written as
\begin{eqnarray}
F^{(k)}(g(z)) & = &
(-1)^{k}k! \sum_{r=1}^{k} \frac{\binom{k-1}{r-1}}{r!} \frac{(z^{2}+1)^{k+r}}
{z^{k+r}} \psi_{r} \left( \frac{z^{2}+1}{z} \right)    \label{Fk} \\
 &  & + (-1)^{k-1}(k-1)!  \frac{(z^{2}+1)^{k}}{z^{k}},  \ \ \mbox{for} \ k\geq 1 ,\notag
\end{eqnarray}
\noindent
and
\begin{equation} \label{Pnk}
P_{n,k}(g(z)) = z^{k-n}
\sum_{j=0}^{k} \frac{(-1)^{k-j} \binom{k}{j} n!}{(z^{2}+1)^{k-j}}
\, \sum_{r=0}^{\infty} (-1)^{r} \binom{j+r-1}{r} \binom{2r+j}{n} z^{2r}.
\end{equation}

\noindent
The sum
\[\frac{1}{(2n)!}\sum_{k=1}^{2n} \frac{1}{k!} F^{(k)}(g(z) P_{2n,k}(g(z)),\]
 with $F^{(k)}(g(z))$ and $ P_{2n,k}(g(z))$ given by (\ref{Fk}) and (\ref{Pnk}), is expanded in powers of $z$. The constant term gives an 
expression for $v_{2n}$.

\section{An alternative approach to the valuations of $v_{n}$}
\label{sec-vn}

The result of Theorem \ref{thm:2adic} is discussed here starting from
a recurrence for $z_{n} = 4nv_{2n}$.  Using
Legendre inverse relations found in Table 2.5 of \cite{riordan-1968a},
the formula \eqref{form1-vna1} for $v_{2n}$, namely
\begin{equation}
v_{2n} = (-1)^{n+1} \left[
\frac{1}{n} + \sum_{k=1}^{n} (-1)^{k}
\binom{n+k-1}{n-k} \frac{B_{2k}}{2k} \right],
\label{vn-even}
\end{equation}
is inverted to express the Bernoulli numbers in
terms of the coefficients $v_{n}$.  The authors wish to thank M. Rogers who
pointed us to this inversion in \cite{rogers-2005a}.

\begin{lemma}
\label{inversion}
If
\begin{equation}
\frac{a_{n}}{2n} = \sum_{k=1}^{n} \binom{n+k-1}{n-k} \frac{b_{k}}{2k}
\end{equation}
\noindent
then
\begin{equation}
b_{n} = \sum_{k=1}^{n} (-1)^{n-k} \binom{2n}{n+k} a_{k}.
\end{equation}
\end{lemma}

The inversion formula is used next to obtain a recurrence for a
slight modification of the coefficients $v_{2n}$.

\begin{theorem}
\label{def-zn}
Define $z_{n} = 4nv_{2n} = - 8nB_{2n}^{*}(-1)$. Then
$z_{n}$ satisfies the recurrence
\begin{equation}
z_{n} = 2\binom{2n}{n} - \sum_{k=1}^{n-1} \binom{2n}{n+k} z_{k} - 2B_{2n}.
\label{recu-zn}
\end{equation}
% and initial condition $z_{1} = \tfrac{11}{3}$.
\end{theorem}

\begin{proof}
The inversion formula in Lemma \ref{inversion} is used with
\begin{equation}
a_{n} = 2n \left( (-1)^{n+1}v_{2n} - \frac{1}{n} \right) \text{ and }
b_{n} = (-1)^{n} B_{2n}
\end{equation}
\noindent
to obtain from Theorem \ref{thm-vnform} the relation
\begin{equation}
B_{2n} = \binom{2n}{n} - \sum_{k=1}^{n} \binom{2n}{n+k} 2kv_{2k}.
\end{equation}
\noindent
The result follows from here.
\end{proof}

The classical von Staudt--Clausen theorem shows that
$2B_{2n}$ is a rational number with odd denominator. The
recurrence \eqref{recu-zn} shows the same is valid for $z_{n}$. Therefore
\begin{equation*}
z_{n} \text{ reduced modulo } 2 = \text{numerator of } z_{n}
\text{ reduced modulo } 2.
\end{equation*}

\begin{proposition}
\label{prop-mod2}
The sequence $z_{n}$ reduced modulo $2$ is periodic with basic period
$\{ 1, \, 1, \, 0 \}$.
\end{proposition}
\begin{proof}
The proof is by induction on $n$. The induction hypothesis is that the
pattern $\{1, \, 1, \, 0 \}$ repeats from $1$ to $n-1$.

Reduce the recurrence \eqref{recu-zn}
modulo $2$ to obtain
\begin{equation}
z_{n} \equiv - \sum_{k \equiv 1, \, 2 \bmod 3}^{n-1} \binom{2n}{n+k} - 1.
\label{recc-1}
\end{equation}
\noindent
This may be written as
\begin{equation}
z_{n} \equiv - \sum_{k=1}^{\lf \tfrac{n+1}{3} \rf}
\binom{2n}{n+3k-2} -
\sum_{k=1}^{\lf \tfrac{n}{3} \rf}
\binom{2n}{n+3k-1} - 1.
\label{recc-2}
\end{equation}
\noindent
The proof is divided in three cases according to the residue of $n$ modulo
$3$.  \\

\noindent
\textbf{Case 1}. Assume $n = 3m$. Then \eqref{recc-2} becomes
\begin{eqnarray*}
z_{3m} & \equiv & - \sum_{k=1}^{m} \binom{6m}{3m+3k-2}
- \sum_{k=1}^{m} \binom{6m}{3m+3k-1}  -1 \\
 & = & - \sum_{k=1}^{m} \binom{6m+1}{3m+3k-1} - 1.
\end{eqnarray*}
\noindent
The symmetry of the binomial coefficients shows that
\begin{equation*}
\sum_{k=1}^{m} \binom{6m+1}{3m+3k-1} = \frac{1}{2}
\sum_{k=-m+1}^{m} \binom{6m+1}{3m+3k-1} =
\frac{1}{2} \sum_{k=-\infty}^{\infty} \binom{6m+1}{3m+3k-1},
\end{equation*}
\noindent
since the terms added to form the last sum actually vanish.

The evaluation of the sum
\begin{equation}
F(m) = \sum_{k} \binom{6m+1}{3m+3k-1}
\end{equation}
\noindent
may be achieved by using the WZ-technology as developed
in \cite{petkovsek-1996a}. The authors have used the
implementation of this algorithm provided by Peter Paule at RISC. The
algorithm shows that $F(m)$ satisfies the recurrence
\begin{equation}
-64F(m) + 65F(m+1) - F(m+2) = 0.
\end{equation}
\noindent
The initial conditions $F(1) = 42$ and $F(2) = 2730$ give
\begin{equation}
F(m) = \frac{2}{3}(64^{m}-1).
\end{equation}
\noindent
Therefore
\begin{equation}
\sum_{k=1}^{m} \binom{6m+1}{3m+3k-1} = \frac{1}{3}(64^{m}-1)
\label{sum-10}
\end{equation}
\noindent
and then
\begin{equation}
z_{3m} \equiv - \frac{1}{3} (64^{m} + 2) \equiv 0 \bmod 2.
\end{equation}
\noindent
This completes the induction step in the case $n \equiv 0 \bmod 3$.
The other two cases, $n \equiv 1, \, 2 \bmod 3$, are treated by a similar
procedure. The induction step is complete.

\end{proof}

\begin{corollary}
If $n \equiv 1, \, 2 \bmod 3$, then $\nu_{2}(z_{n}) = 0$.
\end{corollary}
\begin{proof}
The previous theorem shows that the numbers $z_{3m+1}$ and $z_{3m+2}$ have
odd numerators.
\end{proof}

\begin{note}
The method used to obtain the values of $z_{n}$ modulo $2$ does not extend
directly to modulo $4$ and $8$. The corresponding binomial sums satisfy
similar recurrences, but now there are boundary terms and lack of symmetry
prevents the WZ-method to be used effectively.
\end{note}

\no
\textbf{Acknowledgments}. The fourth  author acknowledges the
partial support of NSF-DMS 1112656. The third author is a post-doctoral
fellow funded in part by the same grant.  The authors wish to thank
Larry Glasser for the proof given in Section \ref{sec-glasser},
Karl Dilcher with help in the proof of Proposition \ref{prop:Bmod8},
Christoph Koutschan for providing the expression for $A_{n}$ given in
Theorem \ref{thm-koutschan} and  Matthew Rogers for pointing out the
result stated in Lemma \ref{inversion}. The authors also wish to thank
T. Amdeberhan for his valuable input into this paper.

%\bibliography{../../../AllRef/official}
%\bibliographystyle{plain}
%\end{document}

\end{document}